\documentclass{amsart}
\usepackage{amsmath,amsfonts,amsthm,amssymb,amscd, verbatim, graphicx, color, array}
\usepackage{tikz-cd}
\usepackage[margin=2cm]{geometry}
\usepackage{hyperref}

\newcolumntype{L}[1]{>{\raggedright\let\newline\\\arraybackslash\hspace{0pt}}m{#1}}
\newcolumntype{C}[1]{>{\centering\let\newline\\\arraybackslash}m{#1}}
\newcolumntype{R}[1]{>{\raggedleft\let\newline\\\arraybackslash\hspace{0pt}}m{#1}}

\theoremstyle{thmit} % Numbered and Italic
\newtheorem{thm}{Theorem}

\newtheorem{lem}[thm]{Lemma}

\theoremstyle{definition}

\theoremstyle{remark}

\theoremstyle{thmrm} % Numbered and Roman

\newtheorem*{oldproof}{Proof}
\renewenvironment{proof}[1][{}]{\begin{oldproof}[#1]}{\qed\end{oldproof}}

\newcommand{\GL}{\mathrm{GL}}

\newcommand{\symme}{\mathrm{Sym}}

\newcommand{\alter}{\mathrm{Alt}}

\begin{document}

\title{The character table of a sharply 5-transitive subgroup of $\alter(12)$}%
\author{Nick Gill}
\address{Department of Mathematics,\\ University of South Wales, \\Treforest, CF37 1DL, U.K. \\{ \tt nick.gill@southwales.ac.uk}}

\author{Sam Hughes}
\address{Department of Mathematics, \\University of South Wales, \\Treforest, CF37 1DL, U.K. \\{ \tt just.sam.hughes@gmail.com}}
%\email{just.sam.hughes@gmail.com}

%\thanks{}%
\subjclass[2000]{20C15; 20B20; 20C34}%
\keywords{Mathieu groups, sporadic groups, character table, permutation group, multiply transitive}%
\date{\today}%
%\dedicatory{}%
%\commby{}%
% ----------------------------------------------------------------
% AMS-LaTeX Paper ************************************************
% **** -----------------------------------------------------------
% ----------------------------------------------------------------
\begin{abstract}
In this paper we calculate the character table of a sharply $5$-transitive subgroup of $\alter(12)$, and of a sharply $4$-transitive subgroup of $\alter(11)$. Our presentation of these calculations is new because we make no reference to the sporadic simple Mathieu groups, and instead deduce the desired character tables using only the existence of the stated multiply transitive permutation representations.
\end{abstract}
\maketitle
% ----------------------------------------------------------------

\section{Introduction}

In this paper we prove the following theorem.

\begin{thm}\label{t: main}
\hspace{1cm}
\begin{enumerate}
 \item If $G$ is a sharply $5$-transitive subgroup of $\alter(12)$, then the character table is given by Table~\ref{t: G12}.
 \item If $G$ is a sharply $4$-transitive subgroup of $\alter(11)$, then the character table is given by Table~\ref{t: G11}.
\end{enumerate}
\end{thm}

This theorem is not new -- item (1) is a consequence of the fact that $M_{12}$ is the unique sharply $5$-transitive subgroup of $\alter(12)$, and the fact that the character table of $M_{12}$ is known; item (2) is a consequence of the two analogous statements for $M_{11}$.

Our proof of the theorem is new, however, because it makes no reference whatsoever to the groups $M_{11}$ and $M_{12}$ but, instead, deduces the character table using nothing more than the stated assumptions about sharp multiple transitivity.

Our interest in proving this theorem stems from our study of Frobenius' famous 1904 paper in which (amongst many other things) he calculates the character table of $M_{12}$ and $M_{24}$ \cite{frobenius}. We were curious to understand Frobenius' methods because, in the late nineteenth century, there appeared to be some lingering doubt as to the ``existence'' of the Mathieu groups (that is to say, people questioned whether the permutations that Mathieu wrote down in his original paper \cite{mathieu} generated alternating groups, rather than any genuinely ``new'' groups). As late as 1897, Miller published a paper claiming that $M_{24}$ did not exist \cite{miller}, although he retracted this claim soon after \cite{miller2}.

In fact, studying Frobenius' 1904 paper, it seems that Frobenius was in no doubt as to the existence of the Mathieu groups and, indeed, he uses specific properties of these groups when he calculates their character tables (see the {\tt MathOverflow} discussion on this subject for more detail \cite{mo}). Nonetheless, we were left wondering whether, in principle, Frobenius could have calculated the character table of $M_{12}$ using nothing more than the property of sharp $5$-transitivity -- the main result of this paper confirms that the answer to this question is ``yes''! In particular, note that the proof of Theorem~\ref{t: main} below uses little more than the basics of character theory, all of which would have been available to Frobenius in 1904 -- the results that we make use of are summarized at the start of \S\ref{ch}.

\subsection{Further work}
There are two natural avenues for extending the current work. First, it would be nice if the two occurrences of ``$\alter(12)$'' in the theorem could be replaced by ``$\symme(12)$''. It seems quite plausible that such a theorem could be proved using the methods in the current paper, however it would likely make the conjugacy class calculations considerably more complicated and so we have not investigated this in detail.

Second, one wonders about proving the analogous theorem for $M_{24}$, something like: ``If $G$ is a $5$-transitive subgroup of $\alter(24)$ of order 244823040, then the character table of $G$ is as follows...'' It would be nice to prove such a theorem using only elementary character theory, although one might expect that the details would be rather onerous.

\subsection{Structure, notation and thanks}

In what follows we will consider groups $G_9, G_{10}, G_{11}$ and $G_{12}$. For $i=9,\dots, 12$, the group $G_i$ denotes any sharply $(i-7)$-transitive subgroup of $\alter(i)$ on $i$ points.

In \S\ref{cc}, we deduce the conjugacy class structure of $G_9, G_{10}, G_{11}$ and $G_{12}$; in \S\ref{ch}, we deduce the character table of $G_9, G_{10}, G_{11}$ and $G_{12}$.

All of our notation is entirely standard. We remind the reader that a character of $G$ is called {\bf real} if it takes real values for all elements of $G$; similarly, a conjugacy class $C$ of $G$ is called {\bf real} if any element of $C$ is conjugate to its inverse. The connection between these two concepts is given in Theorem~\ref{ComplexCharsThm}, and will be exploited in \S\ref{ch}.

In the course of this research, we asked a question on the website {\tt MathOverflow}. We are particularly grateful to Frieder Ladisch, whose answer to our question clarified some of the history behind this area of research \cite{mo}. 

\section{Conjugacy classes}\label{cc}

In this section we calculate the conjugacy class structure of $G_9, G_{10}, G_{11}$ and $G_{12}$. Note that, for $i=9,\dots, 12$, $G_i$ is any sharply $(i-7)$-transitive subgroup of $\alter(i)$ on $i$ points.

We label conjugacy classes via the cycle structure of their elements. Where there is more than one conjugacy class with the same cycle structure, they are distinguished with subscript Roman letters -- see, for instance, Table~\ref{cc G9} which lists the conjugacy classes of $G_9$, three of which contain elements of type $4^2$.

\subsection{The conjugacy classes of \texorpdfstring{$G_9$}{G9}}\label{ccG9}

Since $|G_9|=72$, and is sharply $2$-transitive on $9$ points, it is straightforward to see that $G_9$ has a normal regular subgroup $N$ with $N\cong C_3 \times C_3$. Now the stabilizer of a point, call it $H$, must act (by conjugation) regularly on the non-zero elements of $N$ -- then $H$ is a group of order $8$ that is isomorphic to a subgroup of $\GL_2(3)$, and one can check directly that $H\cong Q_8$; thus $G\cong (C_3\times C_3)\rtimes Q_8$.

The conjugacy classes of $G_9$ can now be written down in Table~\ref{cc G9}. Note that all classes are real.

\begin{table}
 \centering
 \begin{tabular}{|c|c|c|c|c|c|c|}
 \hline
Type & $1^{9}$& $2^4$& $3^3$& $4^2_A$& $4^2_B$& $4^2_C$ \\
 \hline
Size & 1& 9& 8& 18& 18& 18 \\
 \hline
\end{tabular}
\caption{Conjugacy class sizes in $G_9$}\label{cc G9}
\end{table}

\subsection{The conjugacy classes of \texorpdfstring{$G_{10}$}{G10}}\label{ccG10}

The group $G=G_{10}$ has order $720$, and elements that fix at least one point must have cycle structure in the list given in Table~\ref{cc G9}, with the possibility of fusion for the elements of cycle type $4^2$.

When one considers the cycle type of fixed-point-free elements of $G$, one must exclude all elements that have powers that fix elements and that are not of a type listed in Table~\ref{cc G9}. One obtains immediately that the only possible new cycle types are $5^2$ and $8^12^1$.

There must be an element $g\in G$ of type $5^2$, since $5$ divides $|G|$. What is more, since $C_{\alter(10)}(g)=\langle g\rangle$, we conclude that $C_{G}(g)=\langle g\rangle$ and so the conjugacy class containing $g$ has size $144$. Let $P=\langle g\rangle$, a Sylow $5$-subgroup of $G_{10}$. Since $C_G(P)=P$, and $N_G(P)/C_G(P)\lesssim {\rm Aut}(P)\cong C_5$. We conclude that $N_G(P)$ has order dividing $20$. Sylow's theorems tell us that $|G:N_G(P)|\equiv 1\pmod 5$ and we conclude that $|N_G(P)|=20$. Since all conjugates of $P$ intersect trivially, we conclude that $G$ contains $144$ elements of order $5$, hence there is precisely one conjugacy class of elements of type $5^2$.

We now consider elements of types listed in Table~\ref{cc G9}. For elements of type $3^3$, we observe first that an element of this type fixes a unique point and, since the stabilizer of a point contains a unique Sylow $3$-subgroup of $G$, we conclude that each element of type $3^3$ lies in a unique Sylow $3$-subgroup of $G$. Now we use the fact that the stabilizer of a point is maximal in $G$ to conclude that there are precisely $10$ Sylow $3$-subgroups and, therefore, $80$ elements of type $3^3$ in $G$.

All remaining elements, of which there are $450$, must be of type $4^2$ or $2^18^1$. Now let $P$ be a Sylow $2$-subgroup of $G$ and observe that $P$ has two orbits, of size $2$ and $8$. Let $\{9,10\}$ be the smaller orbit; then the stabilizer of $9$ is equal to the stabilizer of $10$ and is equal to $Q_8$, a Sylow $2$-subgroup of $G_9$. The points in the other orbit each have stabilizers in $P$ of size $2$, and there are four distinct stabilizers; this leaves four elements which must be fixed-point-free, and hence are of type $2^18^1$. What is more these elements cannot be central in $P$, otherwise $P$ would be abelian. We conclude that there are at most two conjugacy classes in $G$ of elements of type $2^18^1$, and they have size $90$ (since they do not commute with any elements of odd order). 

In fact, it is clear that any Sylow $2$-subgroup of $G_{10}$ is characterized by its orbit of size $2$. This implies, first, that, since the same is true of elements of type $2^18^1$, each element of type $2^18^1$ is in a unique Sylow $2$-subgroup; it implies, second, that there are precisely $45$ Sylow $2$-subgroups of $G_{10}$, and so there are 180 elements of type $2^18^1$, split into two conjugacy classes of size $90$.

Now there are at most three conjugacy classes in $G$ of elements of type $4^2$ containing a total of $270$ elements; since these elements are real in $G_9$, they are real in $G_{10}$ and so the classes have even order. One of these classes must be squares of elements of order $8$, and so at least one conjugacy class has size $90$; the others have size $90$ or $180$. There are, therefore, two possibilities for elements of type $4^2$: three conjugacy classes of size $90$, or two of size $90$ and $180$. Note that, since there are $180$ elements of type $2^18^1$, there are only $90$ elements that are squares of these. Thus if there are three conjugacy classes of size $90$, then two of these must have a centralizer isomorphic to $C_4\times C_2$. It is easy to check, though, that this is not possible, given that all involutions are of type $2^4$.

The conjugacy classes of $G_{10}$ are summarized in Table~\ref{cc G10}. Note that we have asterisked the two conjugacy classes that are not real -- it is clear that they are the only conjugacy classes that have a chance of being non-real; to see that they are not real, simply observe that there are no elements of type $4^2$ in $\alter(8)$ that send an $8$-cycle to its inverse.\footnote{It is well-known that the group $G_{10}$ is, in fact, $M_{10}$, the unique non-split extension of $\alter(6)$. One can compare our enumeration of the conjugacy classes of $G_{10}$ with the enumeration of conjugacy classes of $M_{10}$ that appears in the ATLAS \cite{atlas}; note that our class $4^2_B$ is labeled $4C$ in the ATLAS; similarly the two classes of elements of order $8$ are labeled $8C$ and $D\mbox{**}$ in the ATLAS.}

\begin{table}
 \centering
 \begin{tabular}{|c|c|c|c|c|c|c|c|c|}
 \hline
Type & $1^{10}$& $2^4$& $3^3$& $4^2_A$ & $4^2_B$ & $5^2$ & ${2^18^1_A}^\ast$ & ${2^18^1_B}^\ast$ \\
 \hline
Size & 1& 45& 80& 90& 180& 144& 90& 90 \\
 \hline
\end{tabular}
\caption{Conjugacy class sizes in $G_{10}$}\label{cc G10}
\end{table}

\subsection{The conjugacy classes of \texorpdfstring{$G_{11}$}{G11}}\label{ccG11}

The group $G=G_{11}$ has order $7920$, and elements that fix at least one point must have cycle structure given in Table~\ref{cc G10}, with the possibility, once again, of fusion for the elements of cycle type $4^2$ or $2^18^1$.

Taking into account the same considerations as before, we obtain that any fixed-point-free elements of $G$ must have cycle type $11^1$ or $2^13^16^1$.

An element, $g$, of type $11^1$ is self-centralizing in $\alter(11)$ and hence, also in $G_{11}$. We conclude that a conjugacy class of this type has size $720$. On the other hand a Sylow $11$-subgroup of $\alter(11)$ has a normalizer of size $55$, hence $N_G(\langle g\rangle)$ has size $11$ or $55$. If the former, then one immediately concludes that there are 10 conjugacy classes of this type; this means that there are a total of $720$ elements in $G$ that are not of type $11^1$. Since a point-stabilizer of $G_{11}$ has size $720$ and does not contain any elements of type $11^1$, we conclude that a point-stabilizer of $G_{11}$ is normal, a contradiction. Thus $|N_G(g)|=55$ and we conclude, furthermore, that there are two conjugacy classes in $G$ of type $11^1$.

We know that there is a unique conjugacy class of elements of type $5^2$ in $G_{10}$, hence the same is true in $G_{11}$. What is more these elements are self-centralizing, hence this class has size $1584$.

Now consider an element, $g$, of type $3^3$. It is clear that $C_G(g)$ contains an involution if and only if there is an element of type $2^13^16^1$. If this is {\bf not} the case, then the conjugacy class of type $3^3$ has size $7920/9 = 880$, and it is the only non-trivial conjugacy class that does not have size divisible by $3$. But now $7920-880-1 = 7039$ is not divisible by $3$, and we have a contradiction. Hence we conclude that $|C_G(g)|$ is even, and contains an element of type $2^13^16^1$. Since the Sylow $2$-subgroup of $C_{\alter(11)}(g)$ is of size $2$, we conclude that $|C_G(g)| = 18$ and the conjugacy class of type $3^3$ has size $440$.

Let $h$ be an element of type $2^4$. Then $g$ is central in a Sylow $2$-subgroup of $G$, and is also centralized by an element of type $3^3$; it is easy to check that it is not centralized by a Sylow $3$-subgroup of $G$, hence the conjugacy class of involutions has size $7920/48 = 165$.

Elements of type $2^18^1$ are self-centralizing, so these conjugacy classes have size $990$; any fusion of the two conjugacy classes would have to take place in a Sylow $2$-subgroup and, since this Sylow $2$-subgroup is the same as for $G_{10}$, we know that there are two such conjugacy classes.

The remaining elements are of type $4^2$ and of type $2^13^16^1$ and they make up the remaining $2310$ elements. The size of the respective conjugacy classes is $7920/8$ and $7920/6$ and since these two numbers sum to $2310$ we know that there is a unique conjugacy class of each type.

The conjugacy classes of $G_{11}$ are now written down in Table~\ref{cc G11}.  Note that we have asterisked the four conjugacy classes that are not real. It is clear that all the other classes are real, and it clear that the classes of type $11^1$ are non-real, since they are non-real in $\alter(11)$; similarly, the classes of type $2^18^1$ are non-real since they are non-real in $G_{10}$ and any putative ``reversing element'' in $G_{11}$ would have to lie in $G_{10}$.

\begin{table}
 \centering
 \begin{tabular}{|c|c|c|c|c|c|c|c|c|c|c|c|}
 \hline
Type & $1^{11}$ & $ 2^4$&$ 3^3$&$ 4^2$&$ 5^2$&$ 2^13^16^1$&$ {2^18^1_A}^\ast$&$ {2^18^1_B}^\ast$&$ {11^1_A}^\ast$ & $ {11^1_B}^\ast$\\
 \hline
Size & $1$ &$ 165$&$ 440$&$ 990$&$ 1584$&$ 1320$&$ 990$&$ 990$&$ 720$&$ 720$ \\
 \hline
\end{tabular}
\caption{Conjugacy class sizes in $G_{11}$}\label{cc G11}
\end{table}

\subsection{The conjugacy classes of \texorpdfstring{$G_{12}$}{G12}}\label{ccG12}

The group $G=G_{12}$ has order $95040$, and elements that fix at least one point must have cycle structure given in Table~\ref{cc G11}, with the possibility of fusion for the elements of order $8$ and $11$.

Taking into account the same considerations as before, we obtain that any fixed-point-free elements of $G$ must have cycle type $2^110^1$, $3^19^1$, $4^18^1$, $6^2$, $2^24^2$, $3^4$, $2^6$. In order to exclude elements of type $3^19^1$, we need a lemma:

\begin{lem}
Let $P$ be a Sylow $3$-subgroup of $G$. Then $P\cong He(3)$.
\end{lem}

Here $He(3)$ is the Heisenberg group over the field of order $3$, a group of order $27$. Note that all non-trivial elements of $He(3)$ have order $3$, hence there are no elements of order $9$ in $G$.

\begin{proof}
Observe that $|P|=27$. There are five groups of order $27$, three abelian and two non-abelian.

Let us suppose first that $P$ is abelian. Suppose that $\Lambda$ is an orbit of $P$ in its action on $\Omega$, and let $\lambda\in\Lambda$. Then, since $P$ is abelian, any elements that fixes $\lambda$ must fix every element in $\Lambda$. Since no element of $G$ fixes more than $4$ points, this means that $|\Lambda|=1$ or $3$. The group $P$ must not fix more than $4$ points, thus there are at least two orbits of $P$ of size $3$, call these $\Lambda_1$ and $\Lambda_2$. Now the orbit-stabilizer theorem asserts that $P$ has a subgroup $P_1$, of order $9$, that fixes every element of $\Lambda_1$; similarly $P$ has a subgroup $P_2$, of order $9$, that fixes every element of $\Lambda_2$. But now $P_1\cap P_2$ is non-trivial (by order considerations) and an element in the intersection fixes at least the 6 points of $\Lambda_1\cup \Lambda_2$. This is a contradiction.

Suppose, then, that $P$ is the non-abelian group of order $27$ that is not $He(3)$. This is the extraspecial group of exponent $9$; it has center, $Z$, of order $3$; it has a normal elementary-abelian subgroup, $P_0$, of order $9$; and all the elements in $P\setminus P_0$ are of order $9$. Since all of the elements of order $9$ are fixed-point-free, and since a stabilizer of a point in $G$ has order divisible by $9$, we conclude that $P_0$ stabilizes a point, indeed it must fix $3$ points. But, since the elements of order $3$ in $P$ have cycle type $3^3$ this implies that all elements of $P$ of order $3$ fix the same $3$ points. This is impossible.
\end{proof}

The same argument as before gives two conjugacy classes of self-centralizing elements of type $11^1$. Similarly there is a single conjugacy class of self-centralizing elements of type $2^13^16^1$. In addition there is an easy counting argument that says that the number of elements that fix exactly four points is $\binom{12}{4}\times 7=3465$. There are only two conjugacy classes that do this -- of type $2^4$ and $4^2$; what is more, using the fact that the stabilizer of $4$ points is $Q_8$, we see that there are six times as many elements of type $4^2$ as there are of type $2^4$. This leaves us with $495$ of type $2^4$ and $2970$ of type $4^2$.

Let $g$ be an element of type $5^2$. Then $|C_G(g)|$ is even if and only if there is an element of type $2^110^1$. Suppose that this is not the case -- then $C_G(g)$ is of order $5$ and all other non-trivial conjugacy classes have order divisible by $5$. But $95040 - 1 - 95040/5$ is not divisible by $5$, a contradiction. We conclude that $|C_G(g)|=10$, and there exist elements of type $2^110^1$. What is more $N_G(\langle g\rangle)$ must be a group of size $40$ with $C_G(g)$ a normal cyclic subgroup of order $10$. If $h$ is a generator of $C_G(g)$, then $N_G(\langle g\rangle)$ acts by conjugation on $C_G(g)=\langle h\rangle$ and is transitive on the elements of order $10$ in $\langle h\rangle$; we conclude that there is a unique conjugacy class of elements of type $2^110^1$, and this conjugacy class has size $9504$.

Let us examine the Sylow $3$-subgroup of $G$ in more detail. The orbits of $P$ must be of size $3$ and $9$, and a count of elements in the stabilizers immediately yields that $P$ contains $14$ elements of type $3^3$ and $12$ that are fixed-point-free, i.e. of type $3^4$. Since $14\not\equiv0\pmod 3$, we conclude that the non-trivial central elements of $P$ are of type $3^3$, and the centralizer of an element of type $3^3$ is divisible by $27$. In addition, for $g$ of type $3^3$, $|C_{\alter(12)}(g)|$ is not divisible by $4$, and so we conclude that $|C_{G_{12}}(g)|=54$.

A straightforward counting argument tells us that there are $35310$ fixed-point-free elements in $G$; we currently have $47190$ elements unaccounted for, of which the only type that is not fixed-point-free are those of type $2^18^1$. We conclude that there are $11880$ of these; on the other hand, consulting Table~\ref{cc G10}, there are $90$ elements in conjugacy class $2^18^1_A$ for the stabilizer of any two letters. Thus each conjugacy class of this type has size $\binom{12}{2}\times 90=11880$, and we conclude that there is a unique conjugacy class of this type.

We are left with the fixed-point-free elements of $G$ -- there are $35310$ of these, and they are of types $2^6$, $3^4$, $2^24^2$, $6^2$,  $4^18^1$ and $2^110^1$ (although we do not yet know if all of these occur). We already know that there are $9504$ elements of type $2^110^1$ and this leave us $25806$ elements for the rest.

Our earlier calculations imply that the centralizer of an element of type $3^4$ is divisible by $9$ but not $27$; we conclude that the centralizer is of size $9$, $18$ or $36$. Suppose that $g$ is of type $3^4$, and that $C_G(g)$ is not of size $9$ -- then there is an element of type $2^6$ that centralizes $g$ and we conclude that $G$ contains an element $h$ of type $6^2$. Now $C_G(h)$ is of size at most $12$. But now, since $\frac19=\frac{1}{12}+\frac{1}{36}$, we conclude that in any case there are at least $|G|/9=10560$ elements of type $3^4$ or $6^2$.

The remaining fixed-point-free elements of $G$ all have order a power of $2$. Let us, therefore, examine $P$, a Sylow $2$-subgroup of $G$. Since $P$ is of size $64$, we know that $P$ does not fix any points; thus the orbit structure of $P$ is either $8\!-\!4$ or $8\!-\!2\!-\!2$ (since a Sylow $2$-subgroup of $G_9$ has orbit structure $8\!-\!1$).

Suppose that the orbit structure is $8\!-\!2\!-\!2$. If $\alpha$ is a point in an orbit of size $2$, then $P_\alpha$ is a normal index $2$ subgroup of $P$, and so $P_\alpha$ fixes all points in the orbit. If $\beta$ is in the other orbit, then $P_\beta$ does likewise, and so $P_\alpha\cap P_\beta$ has size at least $16$ and fixes $4$ points, a contradiction.

Thus the orbit structure is $8\!-\!4$. Let $\Lambda$ be the orbit of size $4$, let $\Gamma$ be the orbit of size $8$, and let $K$ be the kernel of the action on $\Lambda$; then $K\cong Q_8$, the Sylow $2$-subgroup of $G_9$. Since $N=N_{\symme(8)}(K)$ has size $192\!=\!64\!\times\!3$, we conclude that $P$ is isomorphic to a Sylow $2$-subgroup of $N$. What is more the natural map $P\to \symme(4)$, given by considering the action of $P$ on $\Lambda$, yields that $P/K\cong D_4$, the dihedral group of order $8$.

Next observe that $C=C_{\symme(8)}(K)$ is isomorphic to $Q_8$, and $Z(C)=Z(K)$; indeed all non-central elements of $C$ and $K$ act on $\Gamma$ as elements of type $4\!-\!4$ that all square to the same element. Then the non-central elements of $C$ must be of type $4\!-\!4\!-2\!-\!2$ in the action on $\Omega$; in particular $C$ induces the normal Klein $4$-subgroup of $\symme(4)$ in the action on $\Lambda$. But now one can easily check that the four cosets of $K$ in $N\setminus C$ all contain an $8$-cycle; since two of these cosets must induce a $4$-cycle on $\Lambda$, we conclude that $G$ contains elements of type $8\!-\!4$. These elements are self-centralizing in $\alter(12)$, and so likewise in $G$; what is more there are not enough elements left unaccounted for to allow for more than one such conjugacy class. Thus there is a unique conjugacy class of type $8-4$ and it has size $|G|/8 = 11880$.

There are, in addition, elements of type $2^24^2$, and these have a centralizer of size at most $32$; again, a count of remaining elements leads us to conclude that there is a unique conjugacy class of elements of this type and it has size $|G|/32=2970$.

At this stage, then, we have $\frac{83}{720}|G|$ elements unaccounted for; these are of type $3^4$, $6^2$ and/or $2^6$, and we know that at least $\frac{80}{720}|G|$ of them are of type $3^4$ or $6^2$. Thus there are at most $|G|/240$ elements of type $2^6$. If there are no elements of type $2^6$ centralized by an element of order $3$, then we conclude that there is a unique conjugacy class of elements of type $2^6$, and it must have size $|G|/320$. But this also means that there are no elements of type $6^2$, and that the elements of type $3^4$ all have centralizers of size $9$. This does not yield the correct number of elements.

We conclude that there are elements of type $2^6$ centralized by elements of order $3$. By counting remaining elements, we conclude that there is a unique conjugacy class of elements of type $6^2$, and it has size $|G|/12$; similarly, there is a unique conjugacy class of elements of type $3^4$, and it has size $|G|/36$.

There are, therefore, $|G|/240$ elements of type $2^6$; let $g$ be one such. Notice, first, that a Sylow $3$-subgroup of $G$ does not have a subgroup of order $9$ for which all elements are of type $3^4$; we conclude that $|C_G(g)|$ is not divisible by $9$. Next, notice by our arguments above, that we can take $g\not\in Z(KC)$; it is now an easy matter to check that $|C_P(g)|\leq 16$, where $P$ is a Sylow $2$-subgroup of $G$. Thus $|C_G(g)|$ has size at most $240$ and we conclude that there is exactly one conjugacy class of elements of type $2^6$.

We summarise what we have worked out in Table~\ref{cc G12}. Note that we have asterisked the two conjugacy classes that are not real; it is clear that these are the only possible conjugacy classes that have a chance of being non-real, and it is equally clear that they {\bf are} non-real, since they are non-real in $\alter(12)$.

{\small
\begin{table}
 \centering
 \begin{tabular}{|c|c|c|c|c|c|c|c|c|c|c|c|c|c|c|c|}
 \hline
Type & $1^{12}$ & $2^4$& $2^6$ & $3^3$ & $3^4$&$ 4^2$&$ 2^24^2$&$ 5^2$& $2^13^16^1$ & $6^2$ & $2^18^1$ & $4^18^1$ & $2^110^1$ & ${11^1_A}^\ast$ & ${11^1_B}^\ast$ \\
 \hline
Size & $1$ & $495$ & $396$ &$1760$ & $2640$ & $2970$& $2970$ &$ 9504$& $15840$ & $7920$ & $11880$ & $11880$ & $9504$ & $8640$ & $8640$\\
 \hline
\end{tabular}
\caption{Conjugacy class sizes in $G_{12}$}\label{cc G12}
\end{table}
}

\section{Character Tables}\label{ch}

In this section we work out the character tables of $G_9$, $G_{10}$, $G_{11}$ and $G_{12}$. To do this we will need nothing more than the basics of ordinary character theory, along with enough information about the ordinary characters of the symmetric group to calculate irreducible characters for $G_{12}$.

As a reminder we note down five results that will be particularly useful in what follows.

\begin{thm}\emph{\cite[p.~342]{JamesLiebeck}}\label{2transChar}
If $G$ acts 2-transitively on a finite set $\Omega$, then the character given by $\chi(g)=|{\rm fix}(g)|-1$ is irreducible.
\end{thm}

\begin{thm}\emph{\cite[p.~196]{JamesLiebeck}}
Let $G$ be a finite group and let $g\in G$.  Let $\chi$ be a faithful irreducible character of $G$, then $\chi^2(g)$ decomposes as the direct sum of a symmetric part $\frac{1}{2}(\chi^2(g)+\chi(g^2))$ and an antisymmetric part $\frac{1}{2}(\chi^2(g)-\chi(g^2))$.
\end{thm}

\begin{thm}\emph{\cite[p.~236]{JamesLiebeck}} \label{InduceThm}
Let $G$ be a finite group, $x\in G$, $H\leq G$, and suppose that $H\cap x^G$ breaks up into $l$ conjugacy classes of $H$. If $\chi$ is a character of $H$, then:
\begin{equation}\label{InduceEQ}
(\chi\uparrow G)(x)=|C_G(x)|\left(\frac{\chi(x_1)}{|C_H(x_1)|}+\dots+\frac{\chi(x_l)}{|C_H(x_l)|} \right)
\end{equation}
where $x_1,\dots,x_l\in H$ are representative elements of the $l$ classes of $H$.
\end{thm}

\begin{thm}\emph{\cite[p.~264]{JamesLiebeck}}\label{ComplexCharsThm}
 The number of real characters of $G$ is the same as the number of real conjugacy classes of $G$.
\end{thm}

Next we recall the definition of the inner product on characters: if $\chi$ and $\phi$ are characters of $G$, then we define
\[
 \langle \chi, \phi\rangle = \sum_{i=1}^k\frac{\chi(g_i)\overline{\phi}(g_i)}{|C_G(g_i)|}.
\]
The next theorem is a reminder of Schur's Orthogonality Relations; in particular it tells us how to use this inner product to recognise irreducible characters.

\begin{thm}\emph{\cite[p.~161]{JamesLiebeck}}\label{OrthRel}
Let $\chi_1,\dots,\chi_k$ be the irreducible characters of $G$ and let $g_1,\dots,g_k$ be representative elements of the conjugacy classes of $G$, then for every $r,s\in\{1,\dots,k\}$ we have:
\begin{align}
\langle \chi_r, \chi_s\rangle&=\delta_{rs}; \label{RowRel}\\
\sum_{i=1}^k\chi_i(g_r)\overline{\chi_i}(g_s)&=\delta_ {rs}|C_G(g_r)|. \label{ColumnRel}
\end{align}

\end{thm}

\subsection{The Character Table of \texorpdfstring{$G_9$}{G9}}\label{chG9}

We know that $G_9$ has 6 conjugacy classes and hence 6 irreducible characters. We also know that $G_9$ has a normal subgroup $N\cong C_3\times C_3$, with $G/N \cong Q_8$. We can lift 3 linear characters and a 2 dimensional character from $Q_8$. The final character $\chi_5$ is given by the 2-transitive action of $M_9$ on 9 points (Theorem~\ref{2transChar}).  The character table is given in Table~\ref{t: G9}.

\begin{table}[h!]
\[\begin{array}{c|cccccc}
G_9&1^9&2^4&3^3&4^2_A&4^2_B&4^2_C\\
\hline
\chi_0&1&1&1&1&1&1\\
\chi_1&1&1&1&-1&1&-1\\
\chi_2&1&1&1&1&-1&-1\\
\chi_3&1&1&1&-1&-1&1\\
\chi_4&2&-2&2&0&0&0\\
\chi_5&8&0&-1&0&0&0
\end{array}\]
 \caption{The character table of $G_{9}$}\label{t: G9}
\end{table}

\subsection{The Character table of \texorpdfstring{$G_{10}$}{G10}}\label{chG10}

Note that $G_{10}$ has 8 conjugacy classes, hence 8 irreducible characters.  Let $\chi_0$ denote the trivial character of $G_{10}$ and let $\chi_2$ denote the irreducible character obtained from the 3-transitive action of $G_{10}$.

\begin{center}\begin{tabular}{C{1cm}|C{0.5cm}C{0.5cm}C{0.5cm}C{0.5cm}C{0.5cm}C{0.5cm}C{0.7cm}C{0.7cm}}
$G_{10}$&$1^{10}$&$2^4$&$3^3$&$4^2_A$&$4^2_B$&$5^2$&$2^18^1_A$&$2^18^1_B$\\
\hline
$\chi_0$&1&1&1&1&1&1&1&1\\
$\chi_2$&9&1&0&1&1&$-1$&$-1$&$-1$\\
\end{tabular}\end{center}

We will now try inducing characters from the point stabilizer $G_9$. Let $\chi$ be a character of $G_9$, then using the centralizer orders and Theorem~\ref{InduceThm} we have:
    \[(\chi\uparrow G_{10})(g)=\begin{cases}
10\chi(g)&\text{if } g\in 1^{10};\\
2\chi(g)&\text{if } g\in 2^4,\ 4^2_A;\\
\chi(g_B)+\chi(g_C)&\text{if } g\in 4^2_B;\\
\chi(g)&\text{if } g\in 3^3;\\
0&\text{otherwise.}
\end{cases}\]
(Here $g_A$ and $g_B$ are elements from the $G_9$-classes $4^2_A$ and $4^2_B$ respectively.)
Let $\chi_A$, $\chi_B$ and $\chi_C$ denote the lifts of the characters $\chi_1$, $\chi_2$ and $\chi_4$ for $G_9$ respectively (see Table~\ref{t: G9}).
\begin{center}\begin{tabular}{C{1cm}|C{0.5cm}C{0.5cm}C{0.5cm}C{0.5cm}C{0.5cm}C{0.5cm}C{0.7cm}C{0.7cm}}
$\chi_A$&10&2&1&$-2$&0&0&0&0\\
$\chi_B$&10&2&1&2&$-2$&0&0&0\\
$\chi_C$&20&$-4$&2&0&0&0&0&0
\end{tabular}\end{center}
Taking inner products of the characters we find that $\chi_A$ is irreducible, $\langle\chi_B,\chi_B\rangle=2$ and $\langle\chi_C,\chi_C\rangle=2$.  To keep with our naming convention, relabel $\chi_A$ to $\chi_4$.  Our next strategy will be to construct the anti-symmetric parts of both $\chi_2^2$ and $\chi_4^2$, we will skip over the symmetric decomposition of these characters because they do not yield information that is useful to our endeavours.  Let $\chi_D$ and $\chi_E$ denote the antisymmetric components of $\chi_2^2$ and $\chi_4^2$ respectively.
\begin{center}\begin{tabular}{C{1cm}|C{0.5cm}C{0.5cm}C{0.5cm}C{0.5cm}C{0.5cm}C{0.5cm}C{0.7cm}C{0.7cm}}
$\chi_D$&36&$-4$&0&0&0&1&0&0\\
$\chi_E$&45&$-3$&0&1&$-1$&0&0&0
\end{tabular}\end{center}

Taking inner products of $\chi_D$ with known irreducible and compound characters we find that $\langle\chi_D,\chi_D\rangle=3$ and $\langle\chi_C,\chi_D\rangle=2$.  Let $\chi_7=\chi_D-\chi_C$; then $\langle\chi_7,\chi_7\rangle=1$ and, since $\chi_7(1)>0$, we conclude that $\chi_7$ is an irreducible character.  Repeating the process with $\chi_E$ we find that $\langle\chi_E,\chi_E\rangle=4$, $\langle\chi_7,\chi_E\rangle=1$, $\langle\chi_B,\chi_E\rangle=1$ and $\langle\chi_C,\chi_E\rangle=2$.  Let $\chi_3=\chi_E-\chi_7-\chi_C=\chi_E-\chi_D$, then $\langle\chi_3,\chi_3\rangle=1$ and, since $\chi_3(1)>0$, we conclude that $\chi_3$ is an irreducible character
\begin{center}\begin{tabular}{C{1cm}|C{0.5cm}C{0.5cm}C{0.5cm}C{0.5cm}C{0.5cm}C{0.5cm}C{0.7cm}C{0.7cm}}
$\chi_3$&9&1&0&1&$-1$&$-1$&1&1\\
$\chi_7$&16&0&$-2$&0&0&1&0&0
\end{tabular}\end{center}

Finally, observe that $\langle\chi_B,\chi_3\rangle=1$.  Then $\chi_1=\chi_B-\chi_3$ is irreducible. Using Theorem~\ref{ComplexCharsThm} and the fact that $G_{10}$ has two non-real classes, we conclude that the remaining two characters occur as a complex conjugate pair. These can, then, be calculated using the Schur orthogonality relations.

\begin{table}[h!]
\begin{center}\begin{tabular}{C{1cm}|C{0.5cm}C{0.5cm}C{0.5cm}C{0.5cm}C{0.5cm}C{0.5cm}C{0.7cm}C{0.7cm}}
$G_{10}$&$1^{10}$&$2^4$&$3^3$&$4^2_A$&$4^2_B$&$5^2$&$2^18^1_A$&$2^18^1_B$\\
\hline
$\chi_0$&1&1&1&1&1&1&1&1\\
$\chi_1$&1&1&1&1&$-1$&1&$-1$&$-1$\\
$\chi_2$&9&1&0&1&1&$-1$&$-1$&$-1$\\
$\chi_3$&9&1&0&1&$-1$&$-1$&1&1\\
$\chi_4$&10&2&1&$-2$&0&0&0&0\\
$\chi_5$&10&$-2$&1&0&0&0&$\omega$&$\overline\omega$\\
$\chi_6$&10&$-2$&1&0&0&0&$\overline\omega$&$\omega$\\
$\chi_7$&16&0&$-2$&0&0&1&0&0
\end{tabular}\end{center}
\caption{The character table of $G_{10}$, where $\omega=\sqrt{-2}$.}\label{M10.Char.Table}
\end{table}

\subsubsection{The Structure of \texorpdfstring{$G_{10}$}{G10}}
Using the character table, the following result can be easily derived about the structure of $G_{10}$.

\begin{thm}
The group $G_{10}$ is a non split extension of $C_2$ by a normal subgroup $K=\ker(\chi_1)$.
\end{thm}
\begin{proof}
As the kernels of characters are normal subgroups we see that $K\triangleleft G_{10}$ and moreover this is the only non trivial proper normal subgroup of $G_{10}$. Observe that $G_{10}/K\cong C_2$. Since there are no involutions in $G_{10}\setminus K$, we conclude that we have a non-split extension.
\end{proof}

\subsection{The Character Table of \texorpdfstring{$G_{11}$}{G11}}\label{chG11}

By assumption $G_{11}$ acts 4-transitively on a set of size 11.  By considering the number of fixed points of each conjugacy class we obtain a 10 dimensional irreducible character $\chi_1$ (Theorem~\ref{2transChar}).
\begin{center}\begin{tabular}{C{1cm}|C{0.5cm}C{0.5cm}C{0.5cm}C{0.5cm}C{0.5cm}C{0.9cm}C{0.6cm}C{0.6cm}C{0.5cm}C{0.5cm}}
$G_{11}$&$1^{11}$&$2^4$&$3^3$&$4^2$&$5^2$&$2^13^16^1$&$2^18^1_A$&$2^18^1_B$&$11^1_A$&$11^1_B$\\
\hline
$\chi_0$&$1$&$1$&$1$&$1$&$1$&$1$&$1$&$1$&$1$&$1$\\
$\chi_1$&$10$&$2$&$1$&$2$&$0$&$-1$&$0$&$0$&$-1$&$-1$\\
\end{tabular}\end{center}
Now, let $\chi_S$ and $\chi_A$ be the symmetric and antisymmetric decomposition of $\chi_1^2$.
\begin{center}\begin{tabular}{C{1cm}|C{0.5cm}C{0.5cm}C{0.5cm}C{0.5cm}C{0.5cm}C{0.9cm}C{0.6cm}C{0.6cm}C{0.5cm}C{0.5cm}}
$\chi_S$&55&7&1&3&0&1&1&1&0&0\\
$\chi_A$&45&-3&0&1&0&0&-1&-1&1&1
\end{tabular}\end{center}
A quick calculation gives $\langle\chi_S,\chi_S\rangle=3$ and $\langle\chi_A,\chi_A\rangle=1$.  Moreover, $\langle\chi_S,\chi_0\rangle=1$ and $\langle\chi_S,\chi_1\rangle=1$.  Define $\chi_8=\chi_A$ and $\chi_7=\chi_S-\chi_0-\chi_1$ and note that $\langle\chi_7,\chi_7\rangle=1$.  Hence we have found two new irreducible characters of $G_{11}$.
\begin{center}\begin{tabular}{C{1cm}|C{0.5cm}C{0.5cm}C{0.5cm}C{0.5cm}C{0.5cm}C{0.9cm}C{0.6cm}C{0.6cm}C{0.5cm}C{0.5cm}}
$\chi_7$&$44$&$4$&$-1$&$0$&$-1$&$1$&$0$&$0$&$0$&$0$\\
$\chi_8$&$45$&$-3$&$0$&$1$&$0$&$0$&$-1$&$-1$&$1$&$1$
\end{tabular}\end{center}

\subsubsection{Induction from \texorpdfstring{$G_{10}$}{G10}}
We will now try inducing characters from the 1-point stabilizer subgroup $G_{10}$.  Let $\chi$ be a character of $G_{10}$; then, using the centralizer orders and Theorem~\ref{InduceThm}, we have:
    \[(\chi\uparrow G_{11})(g)=\begin{cases}
11\chi(g)&\text{if } g\in 1^{10};\\
3\chi(g)&\text{if } g\in 2^4;\\
2\chi(g)&\text{if } g\in 3^3;\\
\chi(g_A)+2\chi(g_B)&\text{if } g\in 4^2;\\
\chi(g)&\text{if } g\in 5^2,\ 2^18^1_A,\ 2^18^1_B ;\\
0&\text{otherwise.}
\end{cases}\]
(Here $g_A$ and $g_B$ are elements from the $G_{10}$-classes $4^2_A$ and $4^2_B$ respectively.)
Let $\chi_4$ be the induced character of the non trivial linear character of $G_{10}$ and let $\chi_9$ be the the anti-symmetric decomposition of $\chi_4^2$.
\begin{center}\begin{tabular}{C{1cm}|C{0.5cm}C{0.5cm}C{0.5cm}C{0.5cm}C{0.5cm}C{0.9cm}C{0.6cm}C{0.6cm}C{0.5cm}C{0.5cm}}
$\chi_4$&$11$&$3$&$2$&$-1$&$1$&$0$&$-1$&$-1$&$0$&$0$\\
$\chi_9$&$55$&$-1$&$1$&$-1$&$0$&$-1$&$1$&$1$&$0$&$0$
\end{tabular}\end{center}
We find that $\langle\chi_4,\chi_4\rangle=1$ and $\langle\chi_9,\chi_9\rangle=1$, hence, they are both irreducible characters of $G_{11}$.

\subsubsection{Schur Orthogonality}
The remaining four characters come in complex conjugate pairs.  We can deduce this by using the fact that elements of cycle type $2^18^1$ and $11^1$ are not real and applying Theorem~\ref{ComplexCharsThm}.  Given that each pair will have the same dimension, we can attempt to calculate the dimension of these 4 remaining characters.  Let the dimension of the first pair be $d_1$ and the dimension of the second pair be $d_2$.

We have $d_1^2+d_1^2+d_2^2+d_2^2=712$, hence $d_1^2+d_2^2=356$.  By an exhaustive search we find that 356 can be expressed as the sum of two squares in exactly one way, that is $356=10^2+16^2$.  It immediately follows that $d_1=10$ and $d_2=16$.

Let $\chi_2(1)=\chi_3(1)=10$ and $\chi_5(1)=\chi_6(1)=16$.  We will now use the column relations to calculate the character values for the remaining conjugacy classes.  We use the fact that $\chi(g)=\overline{\chi(g^{-1})}$, and let the characters take the following values:
\begin{center}\begin{tabular}{C{1cm}|C{0.5cm}C{0.5cm}C{0.5cm}C{0.5cm}C{0.5cm}C{0.9cm}C{0.6cm}C{0.6cm}C{0.5cm}C{0.5cm}}
$\chi_2$&$11$&$x_1$&$x_2$&$x_3$&$x_4$&$x_5$&$x_6$&$\overline{x_6}$&$x_7$&$\overline{x_7}$\\
$\chi_3$&$11$&$x_1$&$x_2$&$x_3$&$x_4$&$x_5$&$\overline{x_6}$&$x_6$&$\overline{x_7}$&$x_7$\\
$\chi_5$&$16$&$y_1$&$y_2$&$y_3$&$y_4$&$y_5$&$y_6$&$\overline{y_6}$&$y_7$&$\overline{y_7}$\\
$\chi_6$&$16$&$y_1$&$y_2$&$y_3$&$y_4$&$y_5$&$\overline{y_6}$&$y_6$&$\overline{y_7}$&$y_7$\\
\end{tabular}\end{center}
Substituting the column containing $x_i$ for $i=1,\dots5$ and the first column of the character table into \eqref{ColumnRel} we can obtain values for $x_i$ and $y_i$.  We shall demonstrate this with $x_1$:
\begin{align}
 && 1+20+2\times10x_1+33+2\times16y_1+176-135-55&=0 \nonumber \\ 
\Longrightarrow && 5x_1+8y_1&=-10. \label{m11orth1}
\end{align}
Substituting the column containing $x_1$ into \eqref{ColumnRel} twice gives:
\begin{align}
&& 1+4+2x_1^2+9+2y_1^2+16+9+1&=48 \nonumber \\
\Longrightarrow && x_1^2+y_1^2&=4. \label{m11orth2}
\end{align}
Solving \eqref{m11orth1} and \eqref{m11orth2}, we obtain two solutions $x_1=-2,\ y_1=0$ and $x_1=\frac{78}{89},\ y_1=\frac{-160}{89}$.  The second set of these cannot be expressed as sum of $2$nd roots of unity.  Hence, $x_1=-2$ and $y_1=0$.  Continuing in this manner we obtain:
\begin{center}\begin{tabular}{C{1cm}|C{0.5cm}C{0.5cm}C{0.5cm}C{0.5cm}C{0.5cm}C{0.9cm}C{0.6cm}C{0.6cm}C{0.5cm}C{0.5cm}}
$\chi_2$&$10$&$-2$&$1$&$0$&$0$&$1$&$x_6$&$\overline{x_6}$&$x_7$&$\overline{x_7}$\\
$\chi_5$&$16$&$0$&$-2$&$0$&$1$&$0$&$y_6$&$\overline{y_6}$&$y_7$&$\overline{y_7}$\\
\end{tabular}\end{center}
The remaining values $x_6,\ x_7,\ y_6,$ and $y_7$ can be calculated by repeat applications of the row and column relations.  We find that $x_6=\sqrt{-2}$, $x_7=-1$, $y_6=0$ and $y_7=\frac{1}{2}(-1+\sqrt{-11})$.  Table~\ref{t: G11} shows the complete character table.

\begin{table}[h!]
\centering
\begin{tabular}{C{1cm}|C{0.5cm}C{0.5cm}C{0.5cm}C{0.5cm}C{0.5cm}C{0.9cm}C{0.6cm}C{0.6cm}C{0.5cm}C{0.5cm}}
$G_{11}$&$1^{11}$&$2^4$&$3^3$&$4^2$&$5^2$&$2^13^16^1$&$2^18^1_A$&$2^18^1_B$&$11^1_A$&$11^1_B$\\
\hline
$\chi_0$&$1$&$1$&$1$&$1$&$1$&$1$&$1$&$1$&$1$&$1$\\
$\chi_1$&$10$&$2$&$1$&$2$&$0$&$-1$&$0$&$0$&$-1$&$-1$\\
$\chi_2$&$10$&$-2$&$1$&$0$&$0$&$1$&$\alpha$&$\overline\alpha$&$-$1&$-1$\\
$\chi_3$&$10$&$-2$&$1$&$0$&$0$&$1$&$\overline\alpha$&$\alpha$&$-$1&$-1$\\
$\chi_4$&$11$&$3$&$2$&$-1$&$1$&$0$&$-1$&$-1$&$0$&$0$\\
$\chi_5$&$16$&$0$&$-2$&$0$&$1$&$0$&$0$&$0$&$\beta$&$\overline\beta$\\
$\chi_6$&$16$&$0$&$-2$&$0$&$1$&$0$&$0$&$0$&$\overline\beta$&$\beta$\\
$\chi_7$&$44$&$4$&$-1$&$0$&$-1$&$1$&$0$&$0$&$0$&$0$\\
$\chi_8$&$45$&$-3$&$0$&$1$&$0$&$0$&$-1$&$-1$&$1$&$1$\\
$\chi_9$&$55$&$-1$&$1$&$-1$&$0$&$-1$&$1$&$1$&$0$&$0$
\end{tabular}
\caption{The character table of $G_{11}$, where $\alpha=\sqrt{-2}$ and $\beta=\frac{1}{2}(-1+\sqrt{-11})$.}\label{t: G11}
\end{table}

\subsection{The Character Table of \texorpdfstring{$G_{12}$}{G12}}\label{chG12}

We begin by noting that $G_{12}$ has 15 conjugacy classes and 15 irreducible characters, one of which is the trivial character $\chi_0$.

\subsubsection{The Permutation Character and Tensor Products}
By assumption, $G_{12}$ acts 5-transitively on a set of size 12.  Hence, by Theorem~\ref{2transChar}, we get the permutation character $\chi_1$.
\begin{center}
\begin{tabular}{C{1cm}|C{0.5cm}C{0.5cm}C{0.5cm}C{0.5cm}C{0.5cm}C{0.6cm}C{0.6cm}C{0.5cm}C{0.9cm}C{0.5cm}C{0.6cm}C{0.65cm}C{0.7cm}C{0.6cm}C{0.6cm}}
$G_{12}$&$1^{12}$&$2^4$&$2^6$&$3^3$&$3^4$&$4^2$&$2^24^2$&$5^2$&$2^13^16^1$&$6^2$&$2^18^1$&$4^18^1$&$2^110^1$&$11^1_A$&$11^1_B$\\
\hline
$\chi_0$&$1$&$1$&$1$&$1$&$1$&$1$&$1$&$1$&$1$&$1$&$1$&$1$&$1$&$1$&$1$\\
$\chi_1$&$11$&$-1$&$3$&$2$&$-1$&$3$&$-1$&$1$&$0$&$-1$&$-1$&$1$&$-1$&$0$&$0$
\end{tabular}
\end{center}
Now, let $\chi_S$ and $\chi_A$ be the symmetric and antisymmetric decomposition of $\chi_1^2$.
\begin{center}\begin{tabular}{C{1cm}|C{0.5cm}C{0.5cm}C{0.5cm}C{0.5cm}C{0.5cm}C{0.6cm}C{0.6cm}C{0.5cm}C{0.9cm}C{0.5cm}C{0.6cm}C{0.65cm}C{0.7cm}C{0.6cm}C{0.6cm}}
$\chi_S$&$66$&$10$&$6$&$3$&$0$&$6$&$2$&$1$&$1$&$0$&$2$&$0$&$1$&$0$&$0$\\
$\chi_A$&$55$&$-1$&$-5$&$1$&$1$&$3$&$-1$&$0$&$-1$&$1$&$-1$&$1$&$0$&$0$&$0$
\end{tabular}\end{center}
A quick calculation gives $\langle\chi_S,\chi_S\rangle=3$ and $\langle\chi_A,\chi_A\rangle=1$.  Moreover, $\langle\chi_S,\chi_0\rangle=1$ and $\langle\chi_S,\chi_1\rangle=1$.  Define $\chi_8=\chi_A$ and $\chi_6=\chi_S-\chi_0-\chi_1$.  We have found two new irreducible characters of $G_{12}$.  Squaring these characters, however, is not a viable plan; the characters obtained have dimensions 2916 and 3025.

\begin{center}\begin{tabular}{C{1cm}|C{0.5cm}C{0.5cm}C{0.5cm}C{0.5cm}C{0.5cm}C{0.6cm}C{0.6cm}C{0.5cm}C{0.9cm}C{0.5cm}C{0.6cm}C{0.65cm}C{0.7cm}C{0.6cm}C{0.6cm}}
$\chi_6$&$54$&$6$&$6$&$0$&$0$&$2$&$2$&$-1$&$0$&$0$&$0$&$0$&$1$&$-1$&$-1$\\
$\chi_8$&$55$&$-1$&$-5$&$1$&$1$&$3$&$-1$&$0$&$-1$&$1$&$-1$&$1$&$0$&$0$&$0$
\end{tabular}\end{center}

\subsubsection{Induction from \texorpdfstring{$G_{11}$}{G11}}
We will now try inducing characters from the subgroup $G_{11}$. Let $\chi$ be a character of $G_{11}$; then using the centralizer orders and Theorem~\ref{InduceThm} we have:
\[(\chi\uparrow G_{12})(g)=\begin{cases}
12\chi(g)&\text{if } g\in 1^{12};\\
4\chi(g)&\text{if } g\in 2^4,4^2;\\
3\chi(g)&\text{if } g\in 3^3;\\
2\chi(g)&\text{if } g\in 5^2;\\
\chi(g)&\text{if } g\in 2^13^16^1,11^1_A,11^1_B;\\
\chi(g_A)+\chi(g_B)&\text{if } g\in 2^18^1;\\
0&\text{otherwise.}
\end{cases}\]
(Here $g_A$ and $g_B$ are elements from the $G_{11}$-classes $2^18^1_A$ and $2^18^1_B$ respectively.)
Inducing the trivial character of $G_{11}$ gives a character equal to $\chi_0+\chi_1$.  Inducing the integer valued 10-dimensional character gives a character equal to $\chi_1+\chi_6+\chi_8$.  Now, let $\chi_{12}$ be the induced character of a complex valued 10-dimensional character of $G_{11}$.  We see that $\langle\chi_{12},\chi_{12}\rangle=1$, therefore, $\chi_{12}$ is irreducible.

Let $\chi_V$ be the character obtained by inducing the $55$-dimensional character of $G_{11}$.  Finally, let $\chi_B$ be the induced complex valued 16-dimensional character of $G_{11}$ and let $\omega=\frac{1}{2}(-1+\sqrt{-11})$, then we have:
\begin{center}\begin{tabular}{C{1cm}|C{0.5cm}C{0.5cm}C{0.5cm}C{0.5cm}C{0.5cm}C{0.6cm}C{0.6cm}C{0.5cm}C{0.9cm}C{0.5cm}C{0.6cm}C{0.65cm}C{0.7cm}C{0.6cm}C{0.6cm}}
$\chi_{12}$&$120$&$-8$&$0$&$3$&$0$&$0$&$0$&$0$&$0$&$1$&$0$&$0$&$0$&$-1$&$-1$\\
$\chi_V$&$660$&$-4$&$0$&$3$&$0$&$0$&$-4$&$0$&$0$&$-1$&$2$&$0$&$0$&$0$&$0$\\
$\chi_{B}$&$192$&$0$&$0$&$-6$&$0$&$0$&$0$&$2$&$0$&$0$&$0$&$0$&$0$&$\omega$&$\overline{\omega}$
\end{tabular}\end{center}
Note that $\langle\chi_B,\chi_B\rangle=2$ and $\langle\chi_V,\chi_V\rangle=6$, but the inner product of $\chi_B$ with any known irreducible is $0$.  We have $\langle \chi_V,\chi_{12}\rangle=1$, but the inner product is 0 for any other known irreducibles.  These do not give us any new irreducible characters, but we will use these characters later.

\subsubsection{Restriction from \texorpdfstring{$\symme(12)$}{Sym(12)}}
Using the Frobenius character formula, it is possible to construct low dimensional characters of $\symme(12)$ evaluated over the conjugacy classes of $G_{12}$; recall that these characters are labeled with partitions of $12$. In this section we consider the restriction of some of these characters to the group $G_{12}$; for instance we note that $(\chi_{(11,1)}\downarrow G_{12})=\chi_1$, $(\chi_{(10,2)}\downarrow G_{12})=\chi_2$ and $(\chi_{(10,1,1)}\downarrow G_{12})=\chi_3$. 

The table below gives six new characters that we have constructed in this way. Note that we abuse notation here: given a partition, $\lambda$, of $12$, we would normally write $\chi_\lambda$ for the associated character of $\symme(12)$ whereas here we write $\chi_\lambda$ for the restriction of the associated character to $G_{12}$. Note too that, for ease of notation, we let $\lambda_A=(9,1,1,1)$ and $\lambda_B=(8,1,1,1,1)$.\footnote{The history of this sort of restriction is worth a note: a classical result of Frobenius asserts that if $t$ is a natural number with $t<n/2$, then a subgroup $G<\symme(n)$ is $2t$-transitive if and only if every character of $\symme(n)$ labeled by a partition $(\lambda_1, \lambda_2,\dots, \lambda_a)$ with $\lambda_2+\cdots+\lambda_a\leq t$ remains irreducible when restricted to $G$. This result appears in \cite{frobenius}, the same paper in which Frobenius calculates the character tables of $M_{12}$ and $M_{24}$; indeed Frobenius makes use of this result in his calculation of these tables. Note that Theorem~\ref{2transChar} is a special case of Frobenius' result. A modern version of Frobenius result, making use of the Classification of Finite Simple Groups, was given in a beautiful paper of Saxl \cite{saxl} -- his theorem considers $G$, a subgroup of $\symme(n)$ or of $\alter(n)$ and $\chi$, an ordinary character of $\symme(n)$ or of $\alter(n)$, and he describes all pairs $(\chi, G)$ where the restriction of $\chi$ to $G$ is irreducible.  In our proof we do not make use of Frobenius' (or Saxl's) result -- it is enough for us to be able to calculate the restriction of various characters of $\symme(12)$ directly and naively.}

\begin{center}\begin{tabular}{C{1cm}|C{0.5cm}C{0.5cm}C{0.5cm}C{0.5cm}C{0.5cm}C{0.6cm}C{0.6cm}C{0.5cm}C{0.9cm}C{0.5cm}C{0.6cm}C{0.65cm}C{0.7cm}C{0.6cm}C{0.6cm}}
$G_{12}$&$1^{12}$&$2^4$&$2^6$&$3^3$&$3^4$&$4^2$&$2^24^2$&$5^2$&$2^13^16^1$&$6^2$&$2^18^1$&$4^18^1$&$2^110^1$&$11^1_A$&$11^1_B$\\
\hline
$\chi_{(9,3)}$&$154$&$10$&$-6$&$1$&$4$&$-2$&$-2$&$-1$&$1$&$0$&$0$&$0$&$-1$&$0$&$0$\\
$\chi_{\lambda_A}$&$165$&$-11$&$5$&$3$&$3$&$1$&$1$&$0$&$1$&$-1$&$-1$&$-1$&$0$&$0$&$0$\\
$\chi_{(8,4)}$&$275$&$11$&$15$&$5$&$-4$&$-1$&$3$&$0$&$-1$&$0$&$-1$&$1$&$0$&$0$&$0$\\
$\chi_{(7,5)}$&$297$&$9$&$-15$&$0$&$0$&$5$&$-3$&$2$&$0$&$0$&$-1$&$-1$&$0$&$0$&$0$\\
$\chi_{(3,2,1)}$&$320$&$0$&$0$&$-4$&$-4$&$0$&$0$&$0$&$0$&$0$&$0$&$0$&$0$&$1$&$1$\\
$\chi_{\lambda_B}$&$330$&$-6$&$10$&$6$&$-3$&$-2$&$-2$&$0$&$1$&$0$&$0$&$0$&$0$&$0$&$0$
\end{tabular}\end{center}

We first check the inner product of each character with itself and then with each of the known irreducibles.  We find that the $\langle\chi_{\lambda_A},\chi_{\lambda_A}\rangle=2$ and $\langle\chi_{\lambda_A},\chi_{12}\rangle=1$, define $\chi_5=\chi_{\lambda_A}-\chi_{12}$ and note that $\langle\chi_5,\chi_5\rangle=1$.  Hence, we have found a new irreducible character of $G_{12}$.

\begin{center}\begin{tabular}{C{1cm}|C{0.5cm}C{0.5cm}C{0.5cm}C{0.5cm}C{0.5cm}C{0.6cm}C{0.6cm}C{0.5cm}C{0.9cm}C{0.5cm}C{0.6cm}C{0.65cm}C{0.7cm}C{0.6cm}C{0.6cm}}
$\chi_5$&$45$&$-3$&$5$&$0$&$3$&$1$&$1$&$0$&$0$&$-1$&$-1$&$-1$&$0$&$1$&$1$\\
\end{tabular}\end{center}

We find that $\langle\chi_{\lambda_B},\chi_{\lambda_B}\rangle=3$ and that $\langle\chi_{\lambda_B},\chi_{12}\rangle=1$, hence we define $\chi_X=\chi_{\lambda_B}-\chi_{12}$. Similarly we find that $\langle\chi_{(8,4)},\chi_{(8,4)}\rangle=4$ and that $\langle\chi_{(8,4)},\chi_6\rangle=1$, hence we define $\chi_Y=\chi_{(8,4)}-\chi_6$.  Checking the inner products of $\chi_X$ and $\chi_Y$ with themselves and each other we obtain $\langle\chi_X,\chi_X\rangle=2$, $\langle\chi_Y,\chi_Y\rangle=3$ and $\langle\chi_X,\chi_Y\rangle=2$.  Define $\chi_2=\chi_Y-\chi_X$, we find that $\langle\chi_2,\chi_2\rangle=1$; we have found a new irreducible character of $G_{12}$.

\begin{center}\begin{tabular}{C{1cm}|C{0.5cm}C{0.5cm}C{0.5cm}C{0.5cm}C{0.5cm}C{0.6cm}C{0.6cm}C{0.5cm}C{0.9cm}C{0.5cm}C{0.6cm}C{0.65cm}C{0.7cm}C{0.6cm}C{0.6cm}}
$\chi_Y$&$221$&$5$&$9$&$5$&$-4$&$-3$&$1$&$1$&$-1$&$0$&$-1$&$1$&$-1$&$1$&$1$\\
$\chi_X$&$210$&$2$&$10$&$3$&$-3$&$-2$&$-2$&$0$&$1$&$-1$&$0$&$0$&$0$&$1$&$1$\\
\hline
$\chi_2$&$11$&$3$&$-1$&$2$&$-1$&$-1$&$3$&$1$&$0$&$-1$&$-1$&$1$&$-1$&$0$&$0$
\end{tabular}\end{center}

Let $\chi_S$ and $\chi_A$ be the symmetric and antisymmetric decomposition of $\chi_2^2$.  We find that $\chi_S=\chi_0+\chi_2+\chi_{6}$ and $\chi_9=\chi_A$ is a new irreducible.

\begin{center}\begin{tabular}{C{1cm}|C{0.5cm}C{0.5cm}C{0.5cm}C{0.5cm}C{0.5cm}C{0.6cm}C{0.6cm}C{0.5cm}C{0.9cm}C{0.5cm}C{0.6cm}C{0.65cm}C{0.7cm}C{0.6cm}C{0.6cm}}
$\chi_9$&$55$&$-1$&$-5$&$1$&$1$&$3$&$-1$&$0$&$-1$&$1$&$1$&$-1$&$0$&$0$&$0$\\
\end{tabular}\end{center}

We will now check the inner product of every restricted character and $\chi_V$ with our known irreducible characters of $G_{12}$.
\[\begin{array}{|c|cccccccc|}
\hline
\text{Character}&\chi_0&\chi_1&\chi_2&\chi_5&\chi_6&\chi_8&\chi_9&\chi_{12}\\
\hline
\chi_{(9,3)}&0&0&0&0&0&0&0&0\\
\chi_{\lambda_A}&0&0&0&1&0&0&0&1\\
\chi_{(8,4)}&0&0&1&0&1&0&0&0\\
\chi_{(7,5)}&0&1&0&0&1&0&0&0\\
\chi_{(3,2,1)}&0&0&0&0&0&0&0&0\\
\chi_{\lambda_B}&0&0&0&0&0&0&0&1\\
\chi_V&0&0&0&1&0&0&0&1\\
\hline
\end{array}\]
Define the following characters:
\begin{align*}
\chi_C&=\chi_{(9,3)}\\
\chi_D&=\chi_{(7,5)}-\chi_1-\chi_6\\
\chi_E&=\chi_{(3,2,1)}\\
\chi_F&=\chi_{\lambda_B}-\chi_{12}\\
\chi_W&=\chi_V-\chi_5-\chi_{12}
\end{align*}
The values of these are as follows.
\begin{center}\begin{tabular}{C{1cm}|C{0.5cm}C{0.5cm}C{0.5cm}C{0.5cm}C{0.5cm}C{0.6cm}C{0.6cm}C{0.5cm}C{0.9cm}C{0.5cm}C{0.6cm}C{0.65cm}C{0.7cm}C{0.6cm}C{0.6cm}}
$G_{12}$&$1^{12}$&$2^4$&$2^6$&$3^3$&$3^4$&$4^2$&$2^24^2$&$5^2$&$2^13^16^1$&$6^2$&$2^18^1$&$4^18^1$&$2^110^1$&$11^1_A$&$11^1_B$\\
\hline
$\chi_C$&$154$&$10$&$-6$&$1$&$4$&$-2$&$-2$&$-1$&$1$&$0$&$0$&$0$&$-1$&$0$&$0$\\
$\chi_D$&$231$&$7$&$-9$&$-1$&$0$&$-1$&$-1$&$1$&$1$&$0$&$-1$&$-1$&$1$&$0$&$0$\\
$\chi_E$&$320$&$0$&$0$&$-4$&$-4$&$0$&$0$&$0$&$0$&$0$&$0$&$0$&$0$&$1$&$1$\\
$\chi_F$&$210$&$2$&$10$&$3$&$-3$&$-2$&$-2$&$0$&$1$&$-1$&$0$&$0$&$0$&$1$&$1$\\
$\chi_W$&$485$&$5$&$5$&$-1$&$-1$&$-3$&$-3$&$0$&$-1$&$-1$&$1$&$1$&$0$&$1$&$1$
\end{tabular}\end{center}

Taking the inner products of each of these new characters with each other gives the following:
\[\begin{array}{|c|ccccc|}
\hline
&\chi_C&\chi_D&\chi_E&\chi_F&\chi_W\\
\hline
\chi_C&2&1&0&0&1\\
\chi_D&1&2&1&0&1\\
\chi_E&0&1&2&1&2\\
\chi_F&0&0&1&2&2\\
\chi_W&1&1&2&2&4\\
\hline
\end{array}\]

Note, first, that this table of values implies that each of these characters is the sum of distinct irreducibles. Writing these irreducibles as $\alpha, \beta, \gamma,$ and so on, we see immediately that we can write
\begin{align*}
 \chi_C &=\alpha+\beta; \\
 \chi_D &= \beta+\gamma; \\
 \chi_E &= \gamma+\delta; \\
 \chi_F &= \gamma+\epsilon; \\
 \chi_W &=\alpha+\gamma+\delta+\epsilon.
\end{align*}
Now one obtains that $\alpha=\frac12(\chi_W-\chi_F+\chi_C-\chi_D)$. Once we have $\alpha$ it is an easy matter to obtain the other four irreducibles using the equalities just given. We therefore have five new irreducibles which we label as follows:
\[
 \alpha=\chi_{11}, \, \beta=\chi_7, \, \gamma=\chi_{14}, \, \delta=\chi_{13}, \, \epsilon=\chi_{10}.
\]

Finally, we return to the character $\chi_B$ from earlier. We find that $\langle\chi_B,\chi_{14}\rangle=1$, and so we define $\chi_3=\chi_B-\chi_{14}$. Letting $\chi_4$ be the complex conjugate of $\chi_3$, we obtain our final two irreducibles.  The full character table is given in Table~\ref{t: G12}.

\begin{table}%[H]
\centering
\begin{tabular}{C{1cm}|C{0.5cm}C{0.5cm}C{0.5cm}C{0.5cm}C{0.5cm}C{0.6cm}C{0.6cm}C{0.5cm}C{0.9cm}C{0.5cm}C{0.6cm}C{0.65cm}C{0.7cm}C{0.6cm}C{0.6cm}}
$G_{12}$&$1^{12}$&$2^4$&$2^6$&$3^3$&$3^4$&$4^2$&$2^24^2$&$5^2$&$2^13^16^1$&$6^2$&$2^18^1$&$4^18^1$&$2^110^1$&$11^1_A$&$11^1_B$\\
\hline
$\chi_0$&$1$&$1$&$1$&$1$&$1$&$1$&$1$&$1$&$1$&$1$&$1$&$1$&$1$&$1$&$1$\\
$\chi_1$&$11$&$3$&$-1$&$2$&$-1$&$3$&$-1$&$1$&$0$&$-1$&$1$&$-1$&$-1$&$0$&$0$\\
$\chi_2$&$11$&$3$&$-1$&$2$&$-1$&$-1$&$3$&$1$&$0$&$-1$&$-1$&$1$&$-1$&$0$&$0$\\
$\chi_3$&$16$&$0$&$4$&$-2$&$1$&$0$&$0$&$1$&$0$&$1$&$0$&$0$&$-1$&$\omega$&$\overline\omega$\\
$\chi_4$&$16$&$0$&$4$&$-2$&$1$&$0$&$0$&$1$&$0$&$1$&$0$&$0$&$-1$&$\overline\omega$&$\omega$\\
$\chi_5$&$45$&$-3$&$5$&$0$&$3$&$1$&$1$&$0$&$0$&$-1$&$-1$&$-1$&$0$&$1$&$1$\\
$\chi_6$&$54$&$6$&$6$&$0$&$0$&$2$&$2$&$-1$&$0$&$0$&$0$&$0$&$1$&$-1$&$-1$\\
$\chi_7$&$55$&$7$&$-5$&$1$&$1$&$-1$&$-1$&$0$&$1$&$1$&$-1$&$-1$&$0$&$0$&$0$\\
$\chi_8$&$55$&$-1$&$-5$&$1$&$1$&$-1$&$3$&$0$&$-1$&$1$&$-1$&$1$&$0$&$0$&$0$\\
$\chi_9$&$55$&$-1$&$-5$&$1$&$1$&$3$&$-1$&$0$&$-1$&$1$&$1$&$-1$&$0$&$0$&$0$\\
$\chi_{10}$&$66$&$2$&$6$&$3$&$0$&$-2$&$-2$&$1$&$-1$&$0$&$0$&$0$&$1$&$0$&$0$\\
$\chi_{11}$&$99$&$3$&$-1$&$0$&$3$&$-1$&$-1$&$-1$&$0$&$-1$&$1$&$1$&$-1$&$0$&$0$\\
$\chi_{12}$&$120$&$-8$&$0$&$3$&$0$&$0$&$0$&$0$&$1$&$0$&$0$&$0$&$0$&$-1$&$-1$\\
$\chi_{13}$&$144$&$0$&$4$&$0$&$-3$&$0$&$0$&$-1$&$0$&$1$&$0$&$0$&$-1$&$1$&$1$\\
$\chi_{14}$&$176$&$0$&$-4$&$-4$&$-1$&$0$&$0$&$1$&$0$&$-1$&$0$&$0$&$1$&$0$&$0$
\end{tabular}
\caption{The character table of $G_{12}$, where $\omega=\frac{1}{2}(-1+\sqrt{-11})$.}\label{t: G12}
\end{table}

\section{Final remarks}

One can read off many properties of the groups $G_{11}$ and $G_{12}$ by looking at the character tables that we have constructed. Note, for instance, that all of the irreducibles of the two groups have trivial kernel; one concludes immediately that $G_{11}$ and $G_{12}$ are simple.

We saw above, in \S\ref{ccG9}, that $G_9\cong (C_3\times C_3)\rtimes Q_8$. We should note that, although we have not deduced the isomorphism types of $G_{10}$, $G_{11}$ and $G_{12}$, in each case it is well-known that they are unique up to group isomorphism. Indeed $G_{10}\cong M_{10}$, the unique non-split degree 2 extension of $\alter(6)$, while $G_{11}\cong M_{11}$ and $G_{12}\cong M_{12}$, the two smallest sporadic simple groups of Mathieu.

\newcommand{\etalchar}[1]{$^{#1}$}
\providecommand{\bysame}{\leavevmode\hbox to3em{\hrulefill}\thinspace}
\providecommand{\MR}{\relax\ifhmode\unskip\space\fi MR }
\providecommand{\MRhref}[2]{%
  \href{http://www.ams.org/mathscinet-getitem?mr=#1}{#2}
}
\providecommand{\href}[2]{#2}

\end{document}